\documentclass[12pt]{amsart}
\usepackage{url}
\usepackage{pslatex}
\usepackage{amsmath,amsthm,amssymb,amsfonts, stmaryrd}
\usepackage{fancybox}
\usepackage{color, graphicx}

\usepackage[usenames,dvipsnames,svgnames,table]{xcolor}

\newtheorem{theorem}{Theorem}[section]

\newtheorem{conjecture}[theorem]{Conjecture}
\newtheorem{corollary}[theorem]{Corollary}

\newtheorem{notation}[theorem]{Notation}

\newtheorem{proposition}[theorem]{Proposition}

\newcommand{\Z}{\mathbb Z}

\newcommand{\N}{\mathbb N}

\newcommand{\ap}{\textrm{Ap}}

\newcommand{\vs}[1]{\langle #1 \rangle}

\begin{document}

\title{Near-misses in Wilf's conjecture}
\author{Shalom Eliahou and Jean Fromentin}

\address{Shalom Eliahou, Univ. Littoral C\^ote d'Opale, EA 2597 - LMPA - Laboratoire de Math\'ematiques Pures et Appliqu\'ees Joseph Liouville, F-62228 Calais, France and CNRS, FR 2956, France} \email{eliahou@univ-littoral.fr}

\address{Jean Fromentin, Univ. Littoral C\^ote d'Opale, EA 2597 - LMPA - Laboratoire de Math\'ematiques Pures et Appliqu\'ees Joseph Liouville, F-62228 Calais, France and CNRS, FR 2956, France} \email{fromentin@math.cnrs.fr}

\maketitle
\vspace{-1em}
\date{}

\begin{abstract} Let $S \subseteq \N$ be a numerical semigroup with multiplicity $m$, conductor $c$ and minimal generating set $P$. Let $L=S \cap [0,c-1]$ and $W(S)=|P||L|-c$. In 1978, Herbert Wilf asked whether $W(S) \ge 0$ always holds, a question known as Wilf's conjecture and open since then. A related number $W_0(S)$, satisfying $W_0(S) \le W(S)$, has recently been introduced. We say that $S$ is a \emph{near-miss in Wilf's conjecture} if $W_0(S)<0$. Near-misses are very rare. Here we construct infinite families of them, with $c=4m$ and $W_0(S)$ arbitrarily small, and we show that the members of these families still satisfy Wilf's conjecture.
\end{abstract}

\smallskip
\noindent
\keywords{\textbf{Keywords.} Numerical semigroup; conductor; Ap\'ery element; Sidon set; additive combinatorics.}

\medskip

\section{Introduction}\label{introduction} 
Let $\N=\{0,1,2,\dots\}$ and $\N_+=\N\setminus\{0\}$. Given rational numbers $a \le b$, we denote $[a,b]=\{z \in \Z \mid a \le z \le b\}$ the \emph{integer interval} they span, and $[a,\infty[\ =\{z \in \Z \mid z \ge a\}$.

Let $S \subseteq \N$ be a \emph{numerical semigroup}, \emph{i.e.} a submonoid  containing $0$ and with finite complement in $\N$. The \emph{genus} of $S$ is $g(S)=|\N \setminus S|$, its \emph{Frobenius number} is $F(S)=\max (\Z \setminus S)$ and its \emph{conductor} is $c = F(S)+1$. Thus $c+\N \subseteq S$, and $c$ is minimal for this property. Let  $S^*= S \setminus \{0\}$. The \emph{multiplicity} of $S$ is $m=\min S^*$. As in~\cite{E}, we shall denote 
\begin{equation}\label{q, rho}
q = \lceil c/m \rceil \ \textrm{ and }\  \rho=qm-c;
\end{equation}
thus $c=qm-\rho$ and $0 \le \rho \le m-1$. An element $a \in S^*$ is \emph{primitive} if it cannot be written as $a=a_1+a_2$ with $a_1,a_2 \in S^*$. As easily seen, the subset $P \subset S^*$ of primitive elements is contained in the integer interval $[m,m+~c-~1]$. Therefore $P$ is finite, and it generates $S$ as a monoid since every nonzero element in $S$ is a sum of primitive elements. It is well-known and easy to see that $P$ is the \emph{unique minimal generating set} of $S$. Its cardinality $|P|$ is known as the \emph{embedding dimension} of $S$. We shall denote by $D \subset S$ the set of \emph{decomposable elements}, \emph{i.e.} $D = S^*+S^*=S^* \setminus P$. See~\cite{R,RG} for extensive information about numerical semigroups.

\subsection{Wilf's conjecture}
Let $L=S \cap [0,c-1]$, the set of elements of $S$ to the left of its conductor $c$, and denote
$$
W(S) = |P||L|-c.
$$
In 1978, Herbert Wilf asked, in equivalent terms, whether the inequality $$W(S) \ge 0$$ always holds~\cite{W}. This question is known as \emph{Wilf's conjecture}. So far, it has only been settled for a few families of numerical semigroups, including the five independent cases $|P| \le 3$, $|L| \le 4$, $m \le 8$, $g \le 60$ and $q \le 3$. See~\cite{Br08,D,DM,E,FGH,FH,K,MS,Sa} for more details.

\subsection{The number $W_0(S)$}
Denote $S_q=[c,c+m-1]$ and $D_q = D \cap S_q = S_q \setminus P$. We may now define the closely related number $W_0(S)$ introduced in~\cite{E}. It  involves $|P \cap L|$ rather than $|P|$ as in $W(S)$, as well as $D_q$ and the numbers $q, \rho$ given by \eqref{q, rho}.

\begin{notation} Let $S$ be a numerical semigroup. We set
$$W_0(S)  =  |P \cap L||L| - q|D_q| + \rho.$$
\end{notation}

As we shall see in the next section, we have $W(S) \ge W_0(S)$. In particular, if $W_0(S) \ge 0$, then $S$ satisfies Wilf's conjecture. The case $W_0(S)<0$ seems to be extremely rare. The first instances were discovered in 2015 by the second author while performing an  exhaustive computer check of numerical semigroups up to genus 60. Here is the outcome.

\medskip
\noindent
\textbf{Computational result}. \textit{The more than $10^{13}$ numerical semigroups $S$ of genus $g \le 60$ all satisfy $W_0(S) \ge 0$, with exactly 5 exceptions. These 5 exceptions satisfy $W_0(S)=-1$, $W(S) \ge 35$ and $g \in \{43, 51, 55, 59\}$.}

\medskip
We shall describe these five exceptions in the next section. Prompted by their unexpected existence, we say that $S$ is a \emph{near-miss} in Wilf's conjecture if $W_0(S) < 0$. 

The next instances of near-misses were discovered by Manuel Delgado. More precisely, he proved the following result by explicit construction.
\begin{theorem}[\cite{D}] For any $z \in \Z$, there exist infinitely many numerical semigroups $S$ such that $W_0(S)=z$. 
\end{theorem}
(The number $W_0(S)$ is denoted $E(S)$ in~\cite{D}.) He further proved that all the near-misses in his constructions satisfy Wilf's conjecture.

\smallskip
Our aim in this paper is to explain the structure of the original five near-misses of genus $g \le 60$, construct infinite families of similar ones, and show that again, they still satisfy Wilf's conjecture even though their numbers $W_0(S)$ get arbitrarily small in $\Z$.

\smallskip
Thus, both~\cite{D} and the present paper provide constructions of families of numerical semigroups $S$ such that $W_0(S)$ goes to minus infinity. The main difference is that in~\cite{D}, the cardinality $|P \cap L|$ remains constant at $3$ and $q$ goes to infinity, whereas here, the cardinality $|P \cap L|$ goes to infinity and $q$ remains constant at $4$.   In a sense, the case $q=4$ is best possible, as witnessed by the following result.

\begin{theorem}[\cite{E}] Let $S$ be a numerical semigroup such that $q \le 3$. Then $W_0(S) \ge 0$.
\end{theorem}
Hence Wilf's conjecture holds for $q \le 3$. For $q=1$ this is trivial, and for $q=2$ this was first shown in ~\cite{K}. Informally, most numerical semigroups satisfy $q \le 3$, as proved by Zhai in~\cite{Z}. Combining these results, it follows that \emph{Wilf's conjecture is asymptotically true as the genus goes to infinity}.

\subsection{Contents}
In Section~\ref{basic}, we describe the original near-misses of genus $g \le 60$, we recall some basic notions and notation, and we compare the numbers $W(S)$ and $W_0(S)$. In Section~\ref{constructions} we construct, for any integer $n \ge 3$, a numerical semigroup $S$ for which $q=4$, $|P \cap L|=n$ and $W_0(S) = -\binom{n}{3}$. We start with the case $n=3$, and then generalize it to $n\ge 4$ using the notion of $B_h$ sets from additive combinatorics, specifically for $h=3$. In Section~\ref{satisfying Wilf}, we prove that the numerical semigroups $S$ constructed in Section~\ref{constructions} all satisfy $W(S) \ge 9$. The paper ends with the conjecture that our construction is optimal, in the sense that if $q=4$ and $|P \cap L|=n$, then probably $W_0(S) \ge -\binom{n}{3}$.

\section{Basic notions and notation}\label{basic}
Throughout this section, let $S \subseteq \N$ be a numerical semigroup with multiplicity $m$ and conductor $c$. Recall that $q = \lceil c/m\rceil$ and $\rho =qm-c$.

\subsection{On the near-misses of genus $g \le 60$}

As previously mentioned, up to genus $g \le 60$, there are exactly 5 near-misses in Wilf's conjecture. The following notation will be useful to describe them.
\begin{notation}\label{truncated}
Given positive integers $a_1,\dots,a_n, t$, we denote 
\begin{eqnarray*}
\vs{a_1,\dots,a_n} & = & \N a_1+\dots+\N a_n, \\
\vs{a_1,\dots,a_n}_t & = & \vs{a_1,\dots,a_n} \cup [t,\infty[.
\end{eqnarray*}
\end{notation}
As is well-known, $\vs{a_1,\dots,a_n}$ is a numerical semigroup if and only if $\gcd(a_1,\dots,a_n)=1$. On the other hand, $\vs{a_1,\dots,a_n}_t$ is always a numerical semigroup, even if $a_1,\dots,a_n$ are not globally coprime, and its conductor $c$ satisfies $c \le t$, with equality $c=t$ if and only if $t-1 \notin \vs{a_1,\dots,a_n}$.

\smallskip
The 5 near-misses up to genus 60 are given in Table~\ref{table}. They all satisfy $c=4m$, that is $q=4$ and $\rho=0$.

\begin{center}
\begin{table}
\begin{tabular}{|c|c|c|c|c|c|c|}
\hline
\ $S$ \ & $m$ & $|P|$ & $|L|$ & $g$ & $W_0(S)$ & $W(S)$ \\
\hline \hline
$\vs{14,22,23}_{56}$ & \ $14$ \ & \ $7$ \  & $13$ & 43 & $-1$ & 35 \\
\hline
$\vs{16,25,26}_{64}$ & $16$ & $9$\ & $13$   & 51 & $-1$  & 53   \\
\hline
$\vs{17,26,28}_{68} $ & $17$ & $10$ \  & $13$  & 55 & $-1$  & 62 \\
\hline
$\vs{17,27,28}_{68} $ & $17$ & $10$ \ & $13$   & 55 & $-1$  & 62   \\
\hline
$\vs{18,28,29}_{72}$ & $18$ & $11$ \  & $13$  & 59 & $-1$  & 71  \\
\hline
\end{tabular}
\medskip
\caption{All near-misses of genus $g \le 60$}
\label{table}
\end{table}
\end{center}

\subsection{Slicing $\N$} Coming back to our given numerical semigroup $S$, we shall denote $$I_q = [c,c+m-1],$$ the leftmost integer interval of length $m$ contained in $S$. More generally, for any $j \in \N$, let us denote by $I_j$ the translate of $I_q$ by $(j-q)m$. That is,
\begin{eqnarray*}
I_j & = & (j-q)m +[c,c+m-1]\\
& = & [c+(j-q)m,c+(j+1-q)m-1] \\ 
& = & [jm-\rho,(j+1)m-\rho-1].
\end{eqnarray*}
Let us also denote
$$
S_j = S \cap I_{j}.
$$
Observe that $S_j=I_j$ if and only if $j \ge q$. Thus $S_q=I_q$ but $S_j \subsetneq I_j $ for $j < q$. Note also that $S_0=\{0\}$ and that $jm \in S_j$. Finally, for $j \ge 1$, let us denote 
\begin{align*}
P_j &=P \cap S_j, \\ 
D_j &=D \cap S_j=S_j\setminus P_j.
\end{align*}

\subsection{Comparing $W(S)$ and $W_0(S)$}
Since $P \subseteq [m,c+m-1]$ as mentioned earlier, we have 
$$P = P_1 \cup \dots \cup P_q.$$
In particular, we have $P \cap L = P_1 \cup \dots \cup P_{q-1}= P \setminus P_q$. The following formula appears in~\cite{E}.
\begin{proposition}\label{W0} We have $W(S) = W_0(S)+|P_q|(|L|-q)$.
\end{proposition}
\begin{proof} By definition, $W(S)=|P||L|-c=|P||L|-qm+\rho$. Now use the two formulas $|P|=|P \cap L|+|P_q|$ and $m=|P_q|+|D_q|$.
\end{proof}

\begin{corollary} If $W_0(S) \ge 0$, then $S$ satisfies Wilf's conjecture.
\end{corollary}
\begin{proof} We have $|L| \ge q$ since $L \supseteq \{0,1,\dots,q-1\}m$. Thus $|P_q|(|L|-q) \ge 0$, implying $W(S) \ge W_0(S)$ by the above proposition.
\end{proof}

\smallskip

Note that if $S$ is a leaf in the tree of all numerical semigroups~\cite{RGGJ1, RGGJ2, Br08}, \emph{i.e.} if $P=P \cap L$, then $P_q=\emptyset$ and so $W_0(S)=W(S)$ by Proposition~\ref{W0}.

\subsection{Ap\'ery elements}\label{subsection Apery}
As customary, let $\ap(S)=\ap(S,m)$ be the set of \emph{Ap\'ery elements of $S$ with respect to $m$}, namely
\begin{eqnarray*}
\ap(S) & = & \{s \in S \mid s-m \notin S\} \\
& = & \{\min(S \cap (i+m\N)) \mid 0 \le i \le m-1\}.
\end{eqnarray*}
Thus $|\ap(S)|=m$, and each element of $\ap(S)$ is the smallest element of its class mod $m$ in $S$. Note that $\min \ap(S)=0$ and $\max \ap(S) = c+m-1$. Note also that $P\setminus \{m\} \subseteq \ap(S)$. 

\smallskip
For convenience, we shall denote $X = \ap(S)$ and $X_i=X \cap S_i$ for all $i \ge 0$. Note then that $X_0=\{0\}$. 

\begin{proposition}\label{formulas for L and Dq} We have
\begin{eqnarray}
|L| & = & q|X_0|+(q-1)|X_1|+\dots+|X_{q-1}|, \label{equation for L} \\
|D_q| & = & |X_0|+|X_1|+\dots+|X_{q-1}|+|X_q \cap D|. \label{equation for D}
\end{eqnarray}
\end{proposition}
\begin{proof} There is the partition
\begin{equation*}
L \ = \bigsqcup_{0 \le i \le q-1} \big(X_i+[0,q-i-1]\cdot m\big).
\end{equation*}
Indeed, for all $0 \le i \le q-1$, all $x \in X_i$ and all $j \ge 0$, we have $X_i+jm \subseteq S_{i+j}$ and 
$$
(x+m\N)\cap L = \{x,x+m,\dots,x+(q-i-1)m\}.
$$
Conversely, every $a \in L$ belongs to a unique subset of this form, where $x \in X$ is uniquely determined by the condition $a \equiv x \bmod m$. This yields the stated partition of $L$. Moreover, we have
$$|\left(X_i+[0,q-i-1]\cdot m\right)|=(q-i)|X_i|.$$ 
Whence formula \eqref{equation for L}. Similar arguments give rise to the decomposition 
\begin{equation}\label{Dq}
D_q \ = \ (X_q \cap D) \ \sqcup \bigsqcup_{0 \le i \le q-1} \big(X_i+(q-i)m\big).
\end{equation}
Whence formula \eqref{equation for D}. 
\end{proof}

\section{Constructions}\label{constructions}

In this section, we construct numerical semigroups $S$ such that $c=4m$ and where $W_0(S)$ is arbitrarily small in $\Z$. We start with a construction yielding infinitely many instances satisfying $W_0(S)=-1$. Then, after recalling the notion of $B_h$ sets from additive combinatorics, we use it to construct, for any $n \ge 3$, infinitely many instances satisfying $W_0(S)=-\binom{n}{3}$.

\subsection{Realizing $W_0(S)=-1$}

First a notation from additive combinatorics. For nonempty subsets $A,B$ of $\Z$ or of any additively written group $G$, denote $A+B=\{a+b \mid a \in A, b \in B\}$ and
$2A=A+A$. More generally, for any $h \in \N_+$, denote $hA=\underbrace{A+\dots+A}_h$.

\begin{proposition}\label{new W0=-1}
Let $m,a,b \in \N_+$ satisfy $(3m+1)/2 \le a < b \le (5m-1)/3$. Let $A = \{a,b\}$, and assume that the elements of 
$$A \cup 2A \cup 3A= \{a,b,2a,a+b,2b,3a,2a+b, a+2b,3b\}$$ 
are pairwise distinct mod $m$. Let $S=\vs{m,a,b}_{4m}$. Then $W_0(S)=-1$.
\end{proposition}
\begin{proof} Note that the inequality $(3m+1)/2 < (5m-1)/3$ implies $m \ge 6$, while the hypothesis on $A \cup 2A \cup 3A$ implies $m \ge 9$. The computation of~$W_0(S)$ requires several steps. 

\medskip
\noindent
\textbf{Claim 1.} \emph{We have}
\begin{equation*}
\begin{array}{rrrrrrr}
m+1 & \le & a & < & b & \le & 2m-2, \\
3m+1 & \le & 2a & < & 2b & \le & 4m-2, \\
4m+1 & \le & 3a & < & 3b & \le & 5m-1.
\end{array}
\end{equation*}

Indeed, these rather loose inequalities follow from the hypotheses on $a,b$. Thus $A \subseteq [m+1,2m-2]$, $2A \subseteq [3m+1,4m-2]$ and $3A \subseteq [4m+1,5m-1]$. 

\medskip
\noindent
\textbf{Claim 2.} \emph{Let $c$ be the conductor of $S$. Then $c=4m$, $q=4$ and $\rho=0$.}

\smallskip
Indeed, since $S=\vs{m,a,b}_{4m}$, we have $c \le 4m$ by construction. By Claim~1, we have 
$$\vs{m,a,b} \cap [3m+1, 4m-1] = (A+2m) \cup 2A \subseteq [3m+1,4m-2].$$
Therefore $4m-1 \notin \vs{m,a,b}$, implying $c=4m$ as desired. Since $q=\lceil c/m \rceil$ and $\rho=qm-c$, we have $q=4, \rho=0$.

\medskip
\noindent
\textbf{Claim 3.} \emph{The elements of $\{0\}\cup A \cup 2A \cup 3A$ are pairwise distinct mod $m$.}

\smallskip
Indeed, the elements of $A \cup 2A \cup 3A$ are pairwise distinct mod $m$ by hypothesis, and it follows from Claim 1 that they are nonzero mod $m$.

\medskip
\noindent
\textbf{Claim 4.} \emph{We have}
$$
X_1  =  A, \ \ X_2  = \emptyset, \ \ X_3  =  2A, \ \ X_4 \cap D  =  3A.
$$

\smallskip
Indeed, it follows from Claim 3 that 
\begin{equation}\label{Ai subset X}
\{0\}\cup A \cup 2A \cup 3A \subseteq X.
\end{equation}
Since $\rho=0$ by Claim 2, we have $I_j=[jm, jm+m-1]$ for all $j \ge 0$. Hence $S_1=S \cap [m, 2m-1]$, $S_2=S \cap [2m, 3m-1]$, $S_3=S \cap [3m, 4m-1]$ and $S_4= [4m, 5m-1]$. Claim 1 then implies $A \subseteq S_1$, $2A \subseteq S_3$ and $3A \subseteq S_4$. On the other hand, since $c=4m$, we have $L=S_0 \cup S_1 \cup S_2 \cup S_3$, and $L \subseteq \vs{m,a,b}$ by construction. Consequently, since $X \cap (S+m)=\emptyset$, we have $X \cap L \subseteq \vs{a,b}$, and similarly $X_4 \cap D \subseteq \vs{a,b}$. Claim 1 then implies $X \cap S_1 \subseteq A$, $X \cap S_2=~\emptyset$, $X \cap S_3 \subseteq 2A$ and $X \cap S_4 \cap D \subseteq 3A$. The fact that these inclusions are equalities follows from \eqref{Ai subset X} and the claim is proved.

\medskip
We are now in a position to compute $W_0(S)=|P \cap L||L|-q|D_q|+\rho$. 
We have $P \cap L=\{m,a,b\}$, $q=4$ and $\rho=0$. Thus $W_0(S)=3|L|-4|D_4|$ here. By Proposition~\ref{formulas for L and Dq}, we have
\begin{eqnarray*}
|L| & = & 4|X_0|+3|X_1|+2|X_2|+|X_3|, \\
|D_4| & = & |X_0|+|X_1|+|X_2|+|X_3|+|X_4 \cap D|. 
\end{eqnarray*}
Of course $X_0=\{0\}$, as noted in Section~\ref{subsection Apery}. Moreover $|X_1|=2$, $|X_2|=0$, $|X_3|=3$ and $|X_4 \cap D|=4$. It follows that $|L|=4 \cdot 1+3\cdot 2+2 \cdot 0+1 \cdot 3=13$ and $|D_4|=1+2+0+3+4=10$. Therefore $W_0(S)=3|L|-4|D_4|=-1$, as stated.
\end{proof}

As an application, we now provide an explicit construction satisfying the hypotheses, and hence the conclusion, of the above result.
\begin{corollary}\label{cor W0=-1} Let $k,m$ be integers such that $k \ge 2$, $m \ge 3k+8$ and $m \equiv k \bmod 2$. Let $a=(3m+k)/2$ and let $S=\vs{m,a,a+1}_{4m}$. Then $W_0(S)=-1$.
\end{corollary}
\begin{proof} Let $b=a+1$ and $A=\{a,b\}$. It suffices to see that the hypotheses of Proposition~\ref{new W0=-1} on $a,b$ and $A$ are satisfied. The inequalities 
\begin{equation}\label{a < b}
(3m+1)/2 \le a < b \le (5m-1)/3
\end{equation}
follow from the hypotheses $k \ge 2$ and $m \ge 3k+8$. It remains to see that the elements of $A \cup 2A \cup 3A$ are pairwise distinct mod $m$. It is equivalent to see that the elements of $(A+3m) \cup (2A+m) \cup 3A$ are pairwise distinct mod $m$. This in turn follows from the chain of inequalities
\begin{align*}
4m+1 &\le 2a+m < a+b+m < 2b+m \\
&< a+3m < b+3m \\
& < 3a < 2a+b < a+2b < 3b \\
&\le 5m-1,
\end{align*}
all straightforward consequences of the hypotheses and \eqref{a < b}.
\end{proof}

\medskip
The five near-misses up to genus $60$ listed in Table~\ref{table}, and satisfying $W_0(S)=-1$, are all covered by the above two results. Indeed, four of them are of the form $S=\vs{m,a,a+1}_{4m}$ and derive from Corollary~\ref{cor W0=-1}, namely with parameters $(m,k)$ = $(14,2), (16,2), (17,3) \textrm{ and }(18,2)$, respectively. The fifth one, that is $\vs{17,26,28}_{68}$, is not of this form but is still covered by Proposition~\ref{new W0=-1}. Indeed, let $m=17, a=26, b=28$. Then these numbers satisfy all the hypotheses of Proposition~\ref{new W0=-1}, namely 
$$(3m+1)/2 \le a < b \le (5m-1)/3,$$ 
and setting $A=\{a,b\}$, we have $2A+m=\{69, 71, 73\}$, $A+3m=\{77,79\}$ and $3A=\{78, 80, 82, 84\}$, showing that the elements of $A \cup 2A \cup 3A$ are pairwise distinct mod $m$, as required.

\medskip
In order to generalize the above construction and get numerical semigroups $S$ with $q=4$ and $W_0(S)$ negative arbitrarily small, we need the notion of $B_h$ sets from additive combinatorics, specifically for $h=3$.

\subsection{$B_h$ sets}\hspace{0mm}
Let $G$ be an abelian group. Let $A \subseteq G$ be a nonempty finite subset, and let $h \ge 1$ be a positive integer. Then
\begin{equation}\label{hA}
|hA| \le \binom{|A|+h-1}{h}.
\end{equation}
See~\cite[Section 2.1]{T}. This upper bound is best understood by noting that the right-hand side counts the number of monomials of degree $h$ in $|A|$ commuting variables.

We say that $A$ is a \emph{$B_h$ set} if equality holds in \eqref{hA}; equivalently, if for all $a_1, \dots,a_h, b_1,\dots,b_h \in A$, we have
$$
a_1+\dots+a_h=b_1+\dots+b_h
$$
if and only if $(a_1,\dots,a_h)$ is a permutation of $(b_1,\dots,b_h)$. See~\cite[Section 4.5]{T}.

\smallskip
Here are some remarks and examples. The property of being a $B_h$ set is stable under translation in $G$. Clearly, every nonempty finite subset of $G$ is a $B_1$ set and, if $h \ge 2$, every $B_h$ set is a $B_{h-1}$ set.

In $G = \Z$, any subset $A=\{a,b\}$ of cardinality $2$ is a $B_h$ set for all $h \ge 1$, since $|hA|=h+1=\binom{|A|+h-1}{h}$. On the other hand, the subset $A=\{3,4,5\} \subset ~\Z$ is not a $B_2$ set since $3+5=4+4$ in $2A$. Note that $B_2$ sets are also called \emph{Sidon sets}.

For any integer $h \ge 2$, there are arbitrarily large $B_h$ sets in $\N_+$. Take for instance $A=\{1,h, h^2, \dots, h^t\}$ for any $t \ge 1$.

Note that a $B_h$ set in $\Z$ does not necessarily induce a $B_h$ set in the group $\Z/m\Z$. However, for any finite subset $A \subset \Z$ and integer $m \ge |A|$, if $A$ induces a $B_h$ set of cardinality $|A|$ in $\Z/m\Z$, then clearly $A$ itself is a $B_h$ set in $\Z$. 

An instance of a $B_3$ set in $\Z/m\Z$ is provided by Proposition~\ref{new W0=-1}. Indeed, given $m,a,b$ and $A=\{a,b\}$ as in that proposition, the hypothesis there on $A \cup 2A \cup 3A$ means that $A$ induces a $B_3$ set in $\Z/m\Z$.

\subsection{Towards arbitrarily small $W_0(S)$}\label{arbitrary}
We now generalize Proposition~\ref{new W0=-1}, allowing for more than 3 left primitive elements, and yielding numerical semigroups $S$, still with $c=4m$, but now with $W_0(S)$ arbitrarily small in $\Z$. The construction requires large $B_3$ sets in $\N_+$.

\begin{proposition}\label{new construction}
Let $m,a,b,n \in \N_+$ satisfy $n \ge 3$ and 
$$(3m+1)/2 \le a < b \le (5m-1)/3.$$ 
Let $A \subset \N_+$ be a subset of cardinality $|A|=n-1$ with $\min A=a$, $\max A=b$ and inducing a $B_3$ set in $\Z/m\Z$. Let $S=\vs{\{m\} \cup A}_{4m}$. Then $W_0(S)=-\binom{n}{3}$.
\end{proposition}
Note that for $n=3$, Proposition~\ref{new construction} exactly reduces to Proposition~\ref{new W0=-1}.
\begin{proof} 
The proof generalizes that of Proposition~\ref{new W0=-1}. For convenience, we repeat  most of the arguments while adapting them to the present context.

\medskip
\noindent
\textbf{Claim 1.} \emph{We have}
$$
\begin{array}{rrrrrrr}
m+1 & \le & a & < & b & \le & 2m-2 \\
3m+1 & \le & 2a & < & 2b & \le & 4m-2 \\
4m+1 & \le & 3a & < & 3b & \le & 5m-1
\end{array}
$$

These are easy consequences of the hypotheses on $a,b$. It follows that $A \subseteq [m+1,2m-2]$, $2A \subseteq [3m+1,4m-2]$ and $3A \subseteq [4m+1,5m-1]$. 

\medskip
\noindent
\textbf{Claim 2.} \emph{Let $c$ be the conductor of $S$. Then $c=4m$, $q=4$ and $\rho=0$.}

\smallskip
Indeed, since $S=\vs{\{m\} \cup A}_{4m}$, we have $c \le 4m$, and equality holds since $4m-1 \notin S$ by Claim 1.

\medskip
\noindent
\textbf{Claim 3.} \emph{The elements of $\{0\}\cup A \cup 2A \cup 3A$ are pairwise distinct mod $m$.}

\smallskip
Indeed, the elements of $A \cup 2A \cup 3A$ are pairwise distinct mod $m$ since $A$ induces a $B_3$ set in $\Z/m\Z$ by hypothesis. Furthermore, it follows from Claim 1 that they are nonzero mod $m$. 

\medskip
\noindent
\textbf{Claim 4.} \emph{We have}
$$
X_1  =  A, \ \ X_2  = \emptyset, \ \ X_3  =  2A, \ \ X_4 \cap D  =  3A.
$$

This directly follows from the preceding claims. See the corresponding point in the proof of Proposition~\ref{new W0=-1}.

\medskip
We may now compute $W_0(S)=|P \cap L||L|-q|D_q|+\rho$. 
We have $P \cap L=\{m\} \cup A$, $q=4$ and $\rho=0$. Hence $|P \cap L|=|A|+1=n$, and so $W_0(S)=n|L|-4|D_4|$ here. By Proposition~\ref{formulas for L and Dq}, we have
\begin{eqnarray*}
|L| & = & 4|X_0|+3|X_1|+2|X_2|+|X_3|, \\
|D_4| & = & |X_0|+|X_1|+|X_2|+|X_3|+|X_4 \cap D|. 
\end{eqnarray*}
Of course $X_0=\{0\}$. Moreover, we have $|X_1|=|A|$, $|X_2|=0$, $|X_3|=|2A|$ and $|X_4 \cap D|=|3A|$. Now, since $A$ is a $B_3$ set of cardinality $|A|=n-1$, in $\Z/m\Z$ and hence in $\Z$, we have 
$$|2A|=\binom{n}{2}, \ \ |3A|=\binom{n+1}{3}.$$
It follows that 
\begin{eqnarray}
|L| & = & 4+3(n-1)+\binom{n}{2} \ \ = \ \ \binom{n}{2}+3n+1, \label{|L|}\\
|D_4| & = & \binom{n-2}{0}+\binom{n-1}{1}+\binom{n}{2}+\binom{n+1}{3} \ \ = \ \ \binom{n+2}{3}. \label{|D|}
\end{eqnarray}
A direct computation then yields $\displaystyle W_0(S)=n|L|-4|D_4|=-\binom{n}{3}$, as desired.
\end{proof}

Here is an application of Proposition~\ref{new construction}. We only need a $B_3$ set $A$ in~$\Z$; the other hypotheses will force $A$ to induce a $B_3$ set in $\Z/m\Z$, as required.
\begin{corollary} Let $n \ge 3$ be an integer. Let $A' \subset \N$ be a $B_3$ set of cardinality $n-1$ containing $0$. Let $r = \max A'$. Let $k,m \in \N_+$ satisfy $k \ge r+1$, $m \ge 3k+6r+2$ and $m \equiv k \bmod 2$. Let $a=(3m+k)/2$ and $A=a+A'$. Let $S=\vs{\{m\} \cup A}_{4m}$. Then $W_0(S)=-\binom{n}{3}$.
 \end{corollary}
\begin{proof} It suffices to see that $A$ satisfies the hypotheses of Proposition~\ref{new construction}. We have $a=\min A$. Let $b=\max A=a+r$. The required inequalities 
\begin{equation}\label{a < b bis}
(3m+1)/2 \le a < b \le (5m-1)/3
\end{equation}
then follow from the hypotheses on $k,m,A$. Of course $A$ is a $B_3$ set, being a translate of the $B_3$ set $A'$. It remains to see that $A$ induces a $B_3$ set of the same cardinality in $\Z/m\Z$. Let 
\begin{eqnarray*}
C & = & A \cup 2A \cup 3A, \\
C' & = & (A+3m) \cup (2A+m) \cup 3A.
\end{eqnarray*}

\noindent
\textbf{Claim.} \emph{$C' \subseteq [4m+1,5m-1]$ and $A+3m$, $2A+m$, $3A$ are pairwise disjoint.}

\smallskip
Indeed, this follows from the chain of inequalities
\begin{align*}
4m+1 &\le 2a+m < a+b+m < 2b+m \\
&< a+3m < b+3m \\
& < 3a < 2a+b < a+2b < 3b \\
&\le 5m-1,
\end{align*}
all straightforward consequences of the hypotheses and \eqref{a < b bis}. Since $A$ is a $B_3$ set, the elements of $C$ are pairwise distinct in $\Z$, and hence so are the elements of $C'$ by the above claim. Moreover, since $C' \subseteq [4m+1,5m-1]$, its elements are also pairwise distinct mod $m$. Hence $A$ is a $B_3$ set mod $m$, as desired.
\end{proof}

Given $n \ge 3$, here is an explicit infinite family of numerical semigroups~$S$ for which $W_0(S)=-\binom{n}{3}$. Let $$A'=\{3^0-1,3-1,\dots,3^{n-2}-1\}.$$ Then $A'$ is a $B_3$ set of cardinality $n-1$ containing $0$, and hence can be used in the above corollary. Let $r=\max A'=3^{n-2}-1$. Let $k$ be any integer such that $k \ge r+1$. Let then $m=3k+6r+2$, $a_k=(3m+k)/2$, $A_k=a_k+A'$ and $S_k=\vs{\{m\} \cup A_k}_{4m}$. Then $W_0(S_k)=-\binom{n}{3}$ for all $k \ge r+1$.

\section{Satisfying Wilf's conjecture}\label{satisfying Wilf}
In this section, we show that all the near-misses constructed above satisfy Wilf's conjecture.

\begin{proposition}\label{satisfaction} Let $S=\vs{\{m\} \cup A}_{4m}$ be a numerical semigroup as constructed in Proposition~\ref{new construction}. Then $W(S) \ge 9$.
\end{proposition}
\begin{proof} We shall freely use any information about $S$ provided in the proof of Proposition~\ref{new construction}. To start with, since $q=4$ here, Proposition~\ref{W0} gives
\begin{equation}\label{W at q=4}
W(S) = |P_4|(|L|-4)+W_0(S).
\end{equation}

\noindent
\textbf{Claim.} We have $|P_4| \ge m/6 \ge |D_4|/6$. 

\smallskip
Indeed, as $S_4=P_4 \sqcup D_4$ and $|S_4|=m$, it follows that $m=|P_4|+|D_4|$. A decomposition of $D_4$ is provided by \eqref{Dq}, namely
\begin{equation}\label{dec D4}
D_4 \ = \ (X_4 \cap D) \ \sqcup \bigsqcup_{0 \le i \le 3} \big(X_i+(4-i)m\big).
\end{equation}
By Claim 4 in the proof of Proposition~\ref{new construction}, we have
$$
X_1  =  A, \ \ X_2  = \emptyset, \ \ X_3  =  2A, \ \ X_4 \cap D  =  3A,
$$
and $X_0=\{0\}$ as always. Injecting this information into \eqref{dec D4} yields
\begin{equation}\label{D4}
D_4 \ = \ \{4m\} \sqcup (A+3m)\sqcup (2A+m) \sqcup 3A.
\end{equation}
By the hypotheses on $a,b$ in~Proposition~\ref{new construction}, and since $a,b$ are integers, we have 
$$
\left\lceil(3m+1)/2\right\rceil  \ \le \ a \ < \  b \ \le \ \left\lfloor(5m-1)/3\right\rfloor.
$$
As easily seen, this implies the following inequalities:
$$
\begin{array}{ccccccc}
m+ \lceil (m+1)/2 \rceil & \le & a & < & b & \le & m+\lfloor (2m-1)/3\rfloor, \\
3m+ 1 & \le & 2a & < & 2b & \le & 3m+\lfloor (m-2)/3\rfloor, \\
4m+ \lceil (m+3)/2 \rceil & \le & 3a & < & 3b & \le & 4m+(m-1).
\end{array}
$$
It follows that
\begin{eqnarray*}
A+3m & \subseteq & 4m+ \left[\lceil(m+1)/2\rceil, \lfloor(2m-1)/3\rfloor\right], \\
2A+m & \subseteq & 4m+ \left[1, \lfloor(m-2)/3\rfloor\right], \\
3A & \subseteq & 4m+ \left[\lceil(m+3)/2\rceil, m-1\right].
\end{eqnarray*}
Now, the point is that these subsets of $S_4=4m+[0,m-1]$ completely avoid the subinterval
\begin{eqnarray*}
J & = & 4m+\left[\lfloor(m-2)/3\rfloor+1, \lceil(m+1)/2\rceil-1 \right] \\
& = & 4m+\left[\lfloor(m+1)/3\rfloor, \lceil(m-1)/2\rceil \right].
\end{eqnarray*}
Thus by \eqref{D4}, we have
$$
D_4 \cap J \ = \ \emptyset.
$$
Since $J \subseteq S_4=P_4 \sqcup D_4$, it follows that $J \subseteq P_4$. Now, as easily seen by considering the six possible classes of $m$ mod 6, we have
$$
|J| \ = \ \lceil(m-1)/2\rceil-\lfloor(m+1)/3\rfloor+1  \ \ge \ m/6
$$
for all $m \in \N_+$. It follows as claimed that $|P_4| \ge m/6$, and also $|P_4| \ge |D_4|/6$ since $m \ge |D_4|$.

\medskip
Plugging the latter estimate on $|P_4|$ into \eqref{W at q=4} yields
\begin{equation}\label{lower}
W(S) \ \ge \ |D_4|(|L|-4)/6+W_0(S).
\end{equation}
By \eqref{|L|}, we have $\displaystyle |L|-4 = \binom{n}{2}+3(n-1)$, where $n=|P \cap L|=|A|+1$. Hence $(|L|-4)/6 \ge 1$ since $n \ge 3$ by assumption. By \eqref{lower}, this implies 
$$
W(S) \ \ge \ |D_4|+W_0(S).
$$
By Proposition~\ref{new construction} and its proof, we have
$$|D_4|=\binom{n+2}{3}, \quad W_0(S) = -\binom{n}{3}.$$
Whence $W(S) \ge 9$, as desired.
\end{proof}

\subsection{Conjectures} For $q=4$, the lower bound on $W_0(S)$ in terms of $|P \cap L|$ provided by Proposition~\ref{new construction} might well be optimal.

\begin{conjecture}\label{conj 1} Let $S$ be a numerical semigroup with $q=4$ and $|P \cap L|=n$. Then $\displaystyle W_0(S) \ge -\binom{n}{3}$.
\end{conjecture}

Here is a more precise formulation.
\begin{conjecture}\label{conj 2} Let $S$ be a numerical semigroup of multiplicity $m$ with $q=4$ and $|P \cap L|=n$. Then the minimum of $W_0(S)-\rho$ should be attained exactly when the following conditions simultaneously hold:
\begin{enumerate}
\item $P \cap L \subseteq S_1$,
\item $X_1$ induces a $B_3$ set in $\Z/m\Z,$
\item $X_2=\emptyset,$ $X_3 = 2X_1$ and $X_4 \cap D=3X_1$.
\end{enumerate}
\end{conjecture}

We leave it to the reader to see that Conjecture~\ref{conj 2} implies Conjecture~\ref{conj 1}, for instance by following the proof of Proposition~\ref{satisfaction}.

\end{document}